\documentclass[12pt]{amsart}
\usepackage{fullpage,url}
\usepackage[all]{xy} 

\DeclareFontEncoding{OT2}{}{} 

\usepackage{color}

\newcommand{\defi}[1]{\textsf{#1}} 

\newcommand{\Aff}{{\mathbb A}}
\newcommand{\C}{{\mathbb C}}
\newcommand{\F}{{\mathbb F}}

\newcommand{\PP}{{\mathbb P}}
\newcommand{\Q}{{\mathbb Q}}
\newcommand{\R}{{\mathbb R}}
\newcommand{\Z}{{\mathbb Z}}

\newcommand{\kbar}{{\overline{k}}}

\newcommand{\calE}{{\mathcal E}}

\newcommand{\calL}{{\mathcal L}}

\newcommand{\calV}{{\mathcal V}}
\newcommand{\calW}{{\mathcal W}}

\newcommand{\OO}{{\mathcal O}}

\DeclareMathOperator{\inv}{inv}

\DeclareMathOperator{\Br}{Br}

\DeclareMathOperator{\ord}{ord}
\DeclareMathOperator{\Sym}{Sym}

\DeclareMathOperator{\Spec}{Spec}

\DeclareMathOperator{\Proj}{Proj}

\newcommand{\sing}{{\operatorname{sing}}}

\newcommand{\injects}{\hookrightarrow}
\newcommand{\isom}{\simeq}

\newcommand{\intersect}{\cap} 
\newcommand{\Union}{\bigcup} 
\newcommand{\tensor}{\otimes}
\newcommand{\directsum}{\oplus} 

\newtheorem{theorem}{Theorem}[section]
\newtheorem{lemma}[theorem]{Lemma}

\newtheorem{proposition}[theorem]{Proposition}

\theoremstyle{definition}

\theoremstyle{remark}
\newtheorem{remark}[theorem]{Remark}

\usepackage[
        backref,
        pdfauthor={Bjorn Poonen}, 
]{hyperref}
\usepackage[alphabetic,backrefs,lite]{amsrefs} 

\begin{document}

\title{Existence of rational points on smooth projective varieties}
\author{Bjorn Poonen}
\thanks{This research was supported by NSF grant DMS-0301280.}
\address{Department of Mathematics, University of California, 
        Berkeley, CA 94720-3840, USA}
\email{poonen@math.berkeley.edu}
\urladdr{http://math.berkeley.edu/\~{}poonen/}
\dedicatory{to Jean-Louis Colliot-Th\'el\`ene on his $60^{\text{th}}$ birthday}
\date{December 11, 2007}

\begin{abstract}
Fix a number field $k$.
We prove that if there is an algorithm for deciding whether a smooth
projective geometrically integral $k$-variety has a $k$-point,
then there is an algorithm for deciding whether an arbitrary $k$-variety
has a $k$-point and also an algorithm for computing $X(k)$
for any $k$-variety $X$ for which $X(k)$ is finite.
The proof involves the construction of a one-parameter algebraic family
of Ch\^atelet surfaces such that exactly one of the surfaces
fails to have a $k$-point.
\end{abstract}

\maketitle

\section{Statement of results}\label{S:statements}

A \defi{variety} over a field $k$ is 
a separated scheme of finite type over $k$.
We will consider algorithms (Turing machines) 
accepting as input $k$-varieties where $k$ is a number field.
Each such variety may be presented by a finite number of affine open
patches together with gluing data, so it admits a finite description 
suitable for input into a Turing machine.
We do not require algorithms to run in polynomial time
or any other specified time, but they must terminate with an answer 
for each allowable input.

\begin{theorem}
\label{T:algorithms}
Fix a number field $k$.
Suppose that there exists an algorithm for deciding whether a
regular projective geometrically integral $k$-variety has a $k$-point.
Then
\begin{enumerate}
\item[(i)]
There is an algorithm for deciding whether an arbitrary $k$-variety
has a $k$-point.
\item[(ii)]
There is an algorithm for computing $X(k)$ for any $k$-variety $X$
for which $X(k)$ is finite.
\end{enumerate}
\end{theorem}

\begin{remark}\hfill
\label{R:many remarks}
\begin{enumerate}
\item[(a)]
For a field $k$ of characteristic $0$, a $k$-variety is regular
if and only if it is smooth over $k$.
Nevertheless, we have two reasons for sometimes 
using the adjective ``regular'':
\begin{itemize}
\item In some situations, for instance when speaking of families of
varieties, it helps to distinguish the absolute notion (regular)
from the relative notion (smooth).
\item In Section~\ref{S:function fields}, we say what can be said
about the analogue for global function fields.
\end{itemize}
\item[(b)]
For regular proper integral $k$-varieties,
the property of having a $k$-point is a birational invariant,
equivalent to the existence of a (not necessarily rank~1) valuation
$v$ on the function field such that $v$ is trivial on $k$
and $k$ maps isomorphically to the residue field:
this follows from \cite{Nishimura1955} 
and also is close to \cite{Lang1954}*{Theorem~3};
see also \cite{CT-Coray-Sansuc1980}*{Lemme~3.1.1}.
Thus one might wonder whether the decision problem is
easier for regular projective geometrically integral varieties 
than for arbitrary ones.
But Theorem~\ref{T:algorithms}(i) 
says that in fact the two problems are equivalent.
\item[(c)]
For $k=\Q$, Theorem~\ref{T:algorithms}(i) was more or less known:
it is easily deduced from 
a result of R.~Robinson \cite{Smorynski1991}*{\S II.7} that 
the problem of deciding the existence of a rational zero of
a polynomial over $\Q$ 
is equivalent to 
the problem of deciding the existence of a nontrivial rational zero 
of a {\em homogeneous} polynomial over $\Q$.
Robinson's argument generalizes easily to number fields with a real place.
\item[(d)]
Theorem~\ref{T:algorithms} becomes virtually trivial if the word
``projective'' is changed to ``affine''.
On the other hand, there are related statements for affine varieties
that are nontrivial: for instance, if there is an algorithm
for deciding whether any irreducible affine plane curve
of geometric genus at least $2$ has a rational point, 
then there is algorithm for determining 
the set of rational points on any such curve~\cite{Kim2003}.
\item[(e)]
By restriction of scalars, if we have an algorithm for deciding
whether a regular projective geometrically integral variety over $\Q$
has a rational point, then we have an analogous algorithm over any
number field.
\item[(f)]
Remark~\ref{R:curve} will imply that to have algorithms in (i) and (ii)
of Theorem~\ref{T:algorithms} for {\em curves},
it would suffice to be able to decide the existence of rational points
on regular projective geometrically integral {\em $3$-folds}.
(If over $\Q$ one uses Robinson's reduction instead, one would need
an algorithm for $9$-folds!)
\end{enumerate}
\end{remark}

Theorem~\ref{T:algorithms} will be deduced in Section~\ref{S:algorithms}
from the following:

\begin{theorem}
\label{T:cover}
Let $X$ be a projective variety over a number field $k$.
Let $U \subseteq X$ be an open subvariety.
Then there exists a regular projective variety $Y$
and a morphism $\pi\colon Y \to X$
such that $\pi(Y(k))=U(k)$.
Moreover, there exists an algorithm
for constructing $(Y,\pi)$ given $(k,X,U)$.
\end{theorem}

The key special case, from which all others will be derived,
is the case where $U=\Aff^1$ and $X=\PP^1$.
In this case we can arrange also for
$\pi^{-1}(t)$ to be smooth and geometrically integral
for all $t \in \PP^1(k)$: see Proposition~\ref{P:Chatelet bundle}.
Thus we will have a family 
of smooth projective geometrically integral varieties
in which every rational fiber but one has a rational point,
an extreme example of geometry {\em not} controlling arithmetic!

\begin{remark}
Theorem~\ref{T:cover} fails for many fields $k$ that are not
number fields, even for those that have a complicated arithmetic.
Proposition~\ref{P:function fields over C}
implies that it fails for the function field of any variety over $\C$,
for instance.
\end{remark}

\section{Notation}\label{S:notation}

Let $k$ be a number field.
Let $\OO_k$ be the ring of integers in $k$.
If $v$ is a place of $k$, let $k_v$ be the completion of $k$ at $v$.
If $v$ is nonarchimedean, let $\F_v$ be the residue field;
call $v$ \defi{odd} if $\#\F_v$ is odd.
If $a \in \OO_k$ generates a prime ideal, let $v_a$ be the associated
valuation, and let $\F_a=\F_{v_a}$.
For $a \in k$, let $a \gg 0$ mean that $a$ is totally positive,
i.e., positive for every real embedding of $k$.

\section{Conic bundles}\label{S:conic bundles}

A \defi{conic} over $k$ is the zero locus in $\PP^2 = \Proj k[x_0,x_1,x_2]$
of a nonzero degree-$2$ homogeneous polynomial in $k[x_0,x_1,x_2]$.
A \defi{conic bundle} $C$ over a $k$-scheme $B$
is the zero locus in $\PP \calE$ 
of a nowhere vanishing section $s$ of $\Sym^2 \calE$,
where $\calE$ is some rank-$3$ vector sheaf on $B$.
We will consider only the special case where
$\calE = \calL_0 \directsum \calL_1 \directsum \calL_2$
for some line sheaves $\calL_i$
and $s=s_0+s_1+s_2$ where $s_i \in \Gamma(B,\calL_i^{\tensor 2})$;
we then call $C \to B$ a \defi{diagonal conic bundle}.

\begin{remark}
\label{R:smooth conic bundle}
If $B$ is a smooth curve over $k$
and $\sum_{i=0}^2 \ord_P(s_i) \le 1$ for every $P \in B$,
then the total space $C$ is smooth over $k$.
\end{remark}

\section{Ch\^atelet surfaces}\label{S:Chatelet}

Fix $\alpha \in k^\times$ 
and $P(x) \in k[x]$ of degree at most $4$.
Let $V_0$ be the affine surface in $\Aff^3$ given by $y^2 - \alpha z^2 = P(x)$.
We want a smooth projective model $V$ of $V_0$.
Define $\tilde{P}(w,x):=w^4 P(x/w)$;
view $\tilde{P}$ as a section of $\OO(4)$ on $\PP^1 := \Proj k[w,x]$.
The construction of Section~\ref{S:conic bundles} with 
$B=\PP^1$, $\calL_0 = \calL_1 = \OO$, $\calL_2 = \OO(2)$,
$s_0:=1$, $s_1:=-\alpha$, and $s_2:=-\tilde{P}$
gives a diagonal conic bundle $V \to \PP^1$
containing $V_0$ as an affine open subvariety.
Since $V \to \PP^1$ is projective, $V$ is projective over $k$ too.
If $P(x)$ is not identically $0$, then $V$ is geometrically integral.
If $P(x)$ is separable and of degree $3$ or $4$,
then $\tilde{P}(w,x)$ is separable and $V$ is smooth over $k$
by Remark~\ref{R:smooth conic bundle};
in this case $V$ is called the \defi{Ch\^atelet surface given by}
$y^2 - \alpha z^2 = P(x)$.

Iskovskikh~\cite{Iskovskikh1971} showed that the Ch\^atelet surface
over $\Q$ given by
\[
        y^2 + z^2 = (x^2-2)(3-x^2)
\]
violated the Hasse principle.
Several years later it was shown that this violation
could be explained by the Brauer-Manin obstruction,
and that more generally, 
any Ch\^atelet surface over a number field 
given by $y^2-az^2=f(x) g(x)$
with $f$ and $g$ distinct irreducible quadratic polynomials 
satisfies the Hasse principle if and only if there is no
Brauer-Manin obstruction \cite{CT-Coray-Sansuc1980}*{Theorem~B}.
Finally, the two-part paper \cites{CT-Sansuc-SD1987I,CT-Sansuc-SD1987II} 
generalized this to all Ch\^atelet surfaces over number fields.

\begin{proposition}
\label{P:Chatelet}
Over every number field $k$, there exists a Ch\^atelet surface $V$
that violates the Hasse principle.
\end{proposition}

The rest of this section is devoted to 
the proof of Proposition~\ref{P:Chatelet},
so a reader interested in only the case $k=\Q$ may accept the
Iskovskikh example and proceed to Section~\ref{S:Chatelet bundles}.
We generalize the argument presented in \cite{Skorobogatov2001}*{p.~145}.

By the Chebotarev density theorem,
we can find $b \in \OO_k$ generating a prime ideal 
such that $b \gg 0$ and $b \equiv 1 \pmod{8 \OO_k}$.
Similarly we find $a \in \OO_k$ generating a prime ideal
such that $a \gg 0$ and $a \equiv 1 \pmod{8 \OO_k}$
and $a$ is a not a square modulo $b$.
We may assume that $\#\F_a,\#\F_b > 5$.
Fix $c \in \OO_k$ such that $b|(ac+1)$.

Let $(x,y)_v \in \{\pm1\}$ be the $v$-adic Hilbert symbol.
Define $(x,y)_b:=(x,y)_{v_b}$.
We will need the following Hilbert symbol calculations later:
\begin{lemma}
\label{L:Hilbert}
We have
\begin{enumerate}
\item[(i)] $(-1,a)_v=1$ for all $v$.
\item[(ii)] $(-1,b)_v=1$ for all $v$.
\item[(iii)] $(ab,a)_b=-1$.
\item[(iv)] $(ab,c)_b=-1$.
\end{enumerate}
\end{lemma}

\begin{proof}\hfill
\begin{enumerate}
\item[(i)] For $v$ archimedean or $2$-adic, it follows
from $a \in k_v^{\times 2}$.
For all other $v$ except $v_a$, it follows from $v(-1)=v(a)=0$.
For $v=v_a$, it follows from the product formula.
\item[(ii)] The proof is the same as that of (i).
\item[(iii)]
By (i), $(ab,a)_b=(-ab,a)_b = (-a,a)_b (b,a)_b = 1 \cdot (b,a)_b = -1$,
by choice of $a$.
\item[(iv)]
Since $b|(ac+1)$, we have $(ab,ac)_b=(ab,-1)_b=(a,-1)_b (b,-1)_b = 1$
by (i) and (ii).
Divide by (iii) to get $(ab,c)_b=-1$. \qedhere
\end{enumerate}
\end{proof}

Let $V$ be the Ch\^atelet surface given by
\begin{equation}
\label{E:Chatelet}
        y^2 - a b z^2 = (x^2 + c)(a x^2 + a c + 1).
\end{equation}
(The quadratic factors on the right are separable and generate the unit ideal
of $k[x]$, so $V$ is smooth over $k$.)

\begin{lemma}
\label{L:local points}
The variety $V$ has a $k_v$-point for every place $v$ of $k$.
\end{lemma}

\begin{proof}
Suppose that $v$ is archimedean or $2$-adic.
Then $ab \in k_v^{\times 2}$, so $V$ has a $k_v$-point.

Suppose that $v$ is odd and $v \notin \{v_a,v_b\}$.
Choose $x \in k$ with $v(x)<0$.
Then the right hand side of \eqref{E:Chatelet}
has even valuation and is hence a norm for the unramified
extension $k_v(\sqrt{ab})/k_v$.
So $V$ has a $k_v$-point.

Suppose that $v=v_b$.
The number of solutions to $y^2=a(x^2+c)$ over $\F_b$
with $x^2+c \ne 0$ and $x \ne 0$
is at least $(\#\F_b+1)-2-2-2 > 0$.
Choose $x \in \OO_k$ reducing to such a solution.
The right hand side of \eqref{E:Chatelet} is congruent
modulo $b$ to $(x^2+c)(ax^2)$, so by Hensel's lemma
it is in $k_v^{\times 2}$.
Thus $V$ has a $k_v$-point.

Suppose that $v=v_a$.
The same argument as in the previous paragraph shows
that we may choose $x \in \OO_k$ such that $x^2+c \in k_v^{\times 2}$.
The other factor $a x^2 + a c + 1$ is $1 \bmod a$, hence in $k_v^{\times 2}$.
Therefore the right hand side of \eqref{E:Chatelet} is in $k_v^{\times 2}$,
so $V$ has a $k_v$-point.
\end{proof}

Let $\kappa(V)$ be the function field of $V$.
Let $A \in \Br \kappa(V)$ be the class of
the quaternion algebra $(ab,x^2+c)$.
Then $A$ equals the class of $(ab,ax^2+ac+1)$ and of $(ab,1+c/x^2)$,
so $A \in \Br V$.
We will show that $A$ gives a Brauer-Manin obstruction to the Hasse 
principle.
For $P_v \in V(k_v)$, 
let $A(P_v) \in \Br k_v$ be the evaluation of $A$ at $P_v$.
Let $\inv_v \colon \Br k_v \injects \Q/\Z$ be the usual invariant map.

\begin{lemma}
\label{L:inv}
For any $P_v \in V(k_v)$, 
\[
\inv_v(A(P_v)) = 
\begin{cases}
0, &\text{ if $v \ne v_b$,} \\
1/2 &\text{ if $v=v_b$.}
\end{cases}
\]
\end{lemma}

\begin{proof}
By continuity, we may assume $P_v \in V_0(k_v)$.
Suppose that $v$ is archimedean or $2$-adic.
Then $ab \in k_v^\times$, so $A$ maps to $0$ in $\Br k_v(V)$.
Hence $\inv_v(A(P_v))=0$.

Suppose that $v$ is odd and $v \notin \{v_a,v_b\}$.
If $v(x)<0$ at $P_v$, then $v(x^2+c)$ is even,
so $\inv_v(A(P_v))=0$.
If $v(x) \ge 0$, then either $x^2+c$ or $ax^2+ac+1$ is a $v$-adic unit,
so using an appropriate description of $A$ shows that $\inv_v(A(P_v))=0$.

Suppose that $v=v_a$.
If $v(x)<0$ at $P_v$, then $x^2+c \in k_v^{\times 2}$,
so $\inv_v(A(P_v))=0$.
If $v(x) \ge 0$, then $a x^2 + a c + 1$ is $1 \bmod a$
so it is in $k_v^{\times 2}$, and again $\inv_v(A(P_v))=0$.

Finally, suppose that $v=v_b$.
If $v(x) \le 0$, then $b|(ac+1)$ implies 
$(ab,ax^2+ac+1)_b=(ab,ax^2)_b=(ab,a)_b = -1$ by Lemma~\ref{L:Hilbert}(iii).
If $v(x)>0$, then $(ab,x^2+c)_b = (ab,c)_b = -1$ by Lemma~\ref{L:Hilbert}(iv).
In either case, $\inv_v(A(P_v))=1/2$.
\end{proof}

Lemma~\ref{L:inv} implies that $V$ has no $k$-point.
This completes the proof of Proposition~\ref{P:Chatelet}.

\section{Ch\^atelet surface bundles}\label{S:Chatelet bundles}

By a \defi{Ch\^atelet surface bundle over $\PP^1$} 
we mean a flat proper morphism
$\calV \to \PP^1$ such that the generic fiber is a Ch\^atelet surface.
For $t \in \PP^1(k)$, we let $\calV_t$ be the fiber above $t$.

We retain the notation of Section~\ref{S:Chatelet}.
Let $\tilde{P}_0(w,x) \in k[w,x]$ be the homogeneous form of degree-$4$
obtained by homogenizing the right hand side of~\eqref{E:Chatelet}.
Let $\tilde{P}_\infty(w,x)$
be any irreducible degree-$4$ form in $k[w,x]$.
Thus $\tilde{P}_0$ and $\tilde{P}_\infty$ are linearly independent.

Let $\calV$ be the diagonal conic bundle over 
$\PP^1 \times \PP^1:= \Proj k[u,v] \times \Proj k[w,x]$
obtained by taking $\calL_0 = \calL_1 := \OO$, $\calL_2 := \OO(1,2)$,
$s_0:=1$, $s_1:=-ab$, and $s_2:=-(u^2 \tilde{P}_\infty + v^2 \tilde{P}_0)$.
Composing $\calV \to \PP^1 \times \PP^1$ with the first projection
$\PP^1 \times \PP^1 \to \PP^1$
lets us view $\calV$ as a Ch\^atelet surface bundle over 
$\PP^1=\Proj k[u,v]$ with projective geometrically integral fibers.
If $u,v \in k$ are not both $0$,
the fiber above $(u:v) \in \PP^1(k)$ is the Ch\^atelet surface given by
\[
        y^2 - a b z^2 = u^2 \tilde{P}_\infty(1,x) + v^2 \tilde{P}_0(1,x),
\]
if smooth over $k$.
In particular, the fiber $\calV_{(0:1)}$ is isomorphic to $V$.

\begin{lemma}
\label{L:thin set}
The set of specializations $(u:v) \in \PP^1(k)$
such that $u^2 \tilde{P}_\infty + v^2 \tilde{P}_0 \in k[w,x]$ is reducible 
(for any or all choices of $(u,v) \in k^2 -\{(0,0)\}$ representing $(u:v)$)
is a thin set in the sense of~\cite{SerreMordellWeil}*{\S9.1}.
\end{lemma}

\begin{proof}
We may assume $u=1$.
The degree-$4$ form $\tilde{P}_\infty + v^2 \tilde{P}_0$ 
over $k(v)$ is irreducible
since it has an irreducible specialization, namely $\tilde{P}_\infty$.
Apply~\cite{SerreMordellWeil}*{\S9.2, Proposition~1}.
\end{proof}

\begin{lemma}
\label{L:fibration}
There exists a finite set $S$ of non-complex places of $k$
and a neighborhood $N_v$ of $(0:1)$ in $\PP^1(k_v)$ for each $v \in S$ 
such that for $t \in \PP^1(k)$ belonging to $N_v$ for each $v \in S$,
the fiber $\calV_t$ has a $k_v$-point for {\em every} $v$.
\end{lemma}

\begin{proof}
This is an application of the ``fibration method'',
which has been used previously in various places
(e.g., \cite{CT-Sansuc-SD1987I}, \cite{Colliot-Thelene1998}*{2.1}, 
\cite{CT-Poonen2000}*{Lemma~3.1}).
Since all geometric fibers of the $k$-morphism $\calV \to \PP^1$ are integral,
the same is true for a model over some ring $\OO_{k,S}$ of $S$-integers.
By adding finitely many $v$ to $S$, 
we can arrange that for nonarchimedean $v \notin S$
the residue field $\F_v$ is large enough 
that every $\F_v$-fiber has a smooth $\F_v$-point by the Weil conjectures;
then by Hensel's lemma any $k_v$-fiber has a $k_v$-point.
Include the real places in $S$,
and exclude the complex places
since for complex $v$ the existence of $k_v$-points on fibers is automatic.
For $v \in S$,
since the fiber above $(0:1)$ has a $k_v$-point,
and since $\calV \to \PP^1$ is smooth above $(0:1)$,
the implicit function theorem implies that 
the image of $\calV(k_v) \to \PP^1(k_v)$
contains a $v$-adic neighborhood $N_v$ of $(0:1)$ in $\PP^1(k_v)$.
\end{proof}

\section{Base change}
\label{S:base change}

The following lemma combines the idea of \cite{CT-Poonen2000}*{Lemma~3.3}
with some new ideas.

\begin{lemma}
\label{L:base change}
Let $P \in \PP^1(k)$.
Let $S$ be a finite set of non-complex places of $k$.
For each $v \in S$, let $N_v$ be a neighborhood of $P$ in $\PP^1(k_v)$.
Let $T$ be a thin subset of $\PP^1(k)$ containing $P$.
Then there exists a $k$-morphism $\gamma\colon \PP^1 \to \PP^1$
such that both of the following hold:
\begin{enumerate}
\item $\gamma(\PP^1(k_v)) \subseteq N_v$ for each $v \in S$.
\item $\gamma^{-1}(T) \intersect \PP^1(k)$
consists of a single point $Q$ with $\gamma(Q)=P$.
\end{enumerate}
\end{lemma}

\begin{proof}
We will construct $\gamma$ as a composition.
But we present the argument as a series of reductions,
each step of which involves taking the inverse image of all the data
under some $\beta \colon \PP^1 \to \PP^1$
and replacing $P$ by some $P' \in \beta^{-1}(P) \intersect \PP^1(k)$.

By definition of thin, 
there exist finitely many regular projective geometrically integral curves 
$C_i$ and morphisms $\nu_i\colon C_i \to \PP^1$ of degree greater than $1$
such that $T \subseteq \Union \nu_i(C_i(k))$.
Choose $Q \in \PP^1(k)$ not equal to a branch point of any $\nu_i$.
Let $\beta\colon \PP^1 \to \PP^1$ be a morphism 
of some large degree $n$ 
such that $\beta$ is totally ramified above $P$ and $Q$
and unramified elsewhere.
Define $\nu_i'\colon C_i' \to \PP^1$ to make the diagram
\[
\xymatrix{
C_i' \ar[d]_{\nu_i'} \ar[r] & C_i \ar[d]^{\nu_i} \\
\PP^1 \ar[r]^{\beta} & \PP^1 \\
}
\]
cartesian.
Since $C_i' \to C_i$ is totally ramified above $\nu_i^{-1}(Q)$,
each $C_i'$ is geometrically integral.
The morphism $C_i \to \PP^1$ must have a branch point $R \in \PP^1(\kbar)$
other than $P$,
and by choice of $Q$ we have $R \ne Q$,
so $\beta^{-1}(R)$ gives $n$ branch points of $\nu_i'\colon C_i' \to \PP^1$.
On the other hand, $\deg \nu_i' = \deg \nu_i$,
so if $n$ is sufficiently large,
the Hurwitz formula implies that the normalization of $C_i'$
has genus greater than $1$.
By Faltings' theorem~\cite{Faltings1983}, $C_i'(k)$ is finite.
We have $\beta^{-1}(T) \subseteq \Union \nu_i'(C_i'(k))$,
so $\beta^{-1}(T)$ is finite.
By pulling all the data back under $\beta$,
we reduce to the case where $T$ is finite.

We may assume that $P=0 \in \PP^1(k)$ with respect to some coordinate.
Then the rational function $t \mapsto 1/(t^2+m)$ maps $\infty$ to $0$
and maps $\PP^1(\R)$ into $N_v$ for each real $v$
if $m \in \Z_{>0}$ is chosen large enough.
Pulling all the data back under the corresponding automorphism of $\PP^1$,
we reduce to the case where $S$ contains no archimedean places.
Similarly, for each nonarchimedean $v \in S$, let $q=\#\F_v$,
choose a rational function $g$ mapping $\{0,1,\infty\}$ to $P$,
and pull back everything under the rational function 
$g(t^m)$ where $m=q^r(q-1)$ for some large $r$;
this lets us replace $S$ by $S-\{v\}$.
Eventually we reduce to the case in which $S=\emptyset$.

For a suitable choice of coordinate, 
$P$ is the point $0 \in \PP^1(k)$, and $\infty \notin T$.
Choose $c \in k^\times$ such that the images of $c$ and $T-\{0\}$
in $k^\times/k^{\times 2}$ do not meet.
Let $\gamma \colon \PP^1 \to \PP^1$ be given by the rational function $ct^2$.
Then $\gamma^{-1}(T) \intersect \PP^1(k)$ consists of the single point $0$.
\end{proof}

\begin{proposition}
\label{P:Chatelet bundle}
There exists a Ch\^atelet surface bundle $\mu\colon \calW \to \PP^1$ over $k$
such that 
\begin{enumerate}
\item[(i)] $\mu$ is smooth over $\PP^1(k)$, and
\item[(ii)] $\mu(\calW(k))=\Aff^1(k)$.
\end{enumerate}
\end{proposition}

\begin{proof}
Obtain $\gamma \colon \PP^1 \to \PP^1$ from Lemma~\ref{L:base change} with 
$P=(0:1)$, with $S$ and $N_v$ as in Lemma~\ref{L:fibration},
with $T$ the thin set of Lemma~\ref{L:thin set};
note that $T$ contains the finitely many $t \in \PP^1(k)$ 
above which $\calV \to \PP^1$ is not smooth.
We may assume that the $Q$ in Lemma~\ref{L:base change} is $\infty$.
Define $\calW$ as the fiber product
\[
\xymatrix{
\calW \ar[d]_\mu \ar[r] & \calV \ar[d] \\
\PP^1 \ar[r]^\gamma & \PP^1. \\
}
\]
Then $\mu$ is smooth above $\PP^1(k)$,
and for every $t \in \PP^1(k)$
the fiber $\calW_t$ has a $k_v$-point for every $v$.
If $t \in \Aff^1(k)$, then $\gamma(t) \notin T$,
so $\calW_t$ is a Ch\^atelet surface defined by an irreducible 
degree-$4$ polynomial, so by \cite{CT-Sansuc-SD1987I}*{Theorem~B(i)(b)} 
$\calW_t$ satisfies the Hasse principle;
thus $\calW_t$ has a $k$-point.
But if $t=\infty$, then $\calW_t$ is isomorphic to $\calV_{(0:1)} \isom V$, 
which has no $k$-point.
\end{proof}

The following proposition will not be needed elsewhere.
Its role is only to illustrate that 
Theorem~\ref{T:cover} and Proposition~\ref{P:Chatelet bundle} 
depend subtly upon properties of $k$: for instance,
they are not true over all fields of cohomological dimension $2$.

\begin{proposition}
\label{P:function fields over C}
Let $k_0$ be an uncountable algebraically closed field,
and let $k$ be a field extension of $k_0$ generated by a set $S$
of cardinality less than $\#k_0$.
Then there is no morphism $\pi\colon \calW \to \PP^1$
of projective $k$-varieties such that $\pi(\calW(k))=\Aff^1(k)$.
\end{proposition}

\begin{proof}
Suppose that $\pi(\calW(k))=\Aff^1(k)$.
Fix a projective embedding $\calW \injects \PP^n_k$.
For each $k_0$-subspace $L$ of $k$ spanned by finitely many finite products
of elements of $S$, let $\calW(L)$ be the set of points 
in $\calW(k) \subseteq \PP^n(k)$
represented by an $(n+1)$-tuple of elements of $L$, not all zero.
The set of points in $\PP^1(k_0) \subseteq \PP^1(k)$
in the image of $\calW(L) \to \PP^1(k)$
is a Zariski closed subset of $\PP^1(k_0)$ not containing $\infty$,
and hence finite.
Taking the union over all $L$, 
we find that the set of points in $\PP^1(k_0)$
in the image of $\calW(k) \to \PP^1(k)$
has cardinality at most $\max\{\#S,\aleph_0\}$,
which is less than $\#k_0$.
In particular, $\pi(\calW(k))$ cannot be a cofinite subset of $\PP^1(k)$.
\end{proof}

\section{Reductions}
\label{S:reductions}

\begin{lemma}
\label{L:P^n}
There exists a projective $k$-variety $Z$ 
and a morphism $\eta\colon Z \to \PP^n$
such that $\eta(Z(k))=\Aff^n(k)$ and $\eta$ is smooth above $\Aff^n(k)$.
\end{lemma}

\begin{proof}
Start with the birational map $(\PP^1)^n \dashrightarrow \PP^n$ 
given by the isomorphism $(\Aff^1)^n \to \Aff^n$.
Resolve the indeterminacy; i.e., find
a projective $k$-variety $J$ and a birational morphism $J \to (\PP^1)^n$ 
whose composition with $(\PP^1)^n \dashrightarrow \PP^n$ extends
to a morphism $J \to \PP^n$ that is an isomorphism above $\Aff^n$.
Define $Z$ to make a cartesian square
\[
\xymatrix{
Z \ar[d] \ar[r] & \calW^n \ar[d]^{\mu^n} \\
J \ar[r] \ar@/_1pc/[rr] & (\PP^1)^n \ar@{-->}[r] & \PP^n \\
}
\]
where $\calW \stackrel{\mu}\to \PP^1$ 
is as in Proposition~\ref{P:Chatelet bundle}.
Let $\eta$ be the composition $Z \to J \to \PP^n$.

By construction of $\calW$, we have $\mu^n(\calW^n(k))=(\Aff^1)^n(k)$,
so the image of $Z(k) \to J(k)$ is contained in the copy of $\Aff^n$ in $J$.
Therefore $\eta(Z(k)) \subseteq \Aff^n(k)$.

On the other hand, if $t \in \Aff^n(k)$,
then $J \to (\PP^1)^n$ is a local isomorphism above $t$,
and $\calW^n \to (\PP^1)^n$ is smooth above $t$,
so $Z \to J$ is smooth above $t$,
and the fiber $\eta^{-1}(t)$ is isomorphic to the corresponding fiber
of $\calW^n \to (\PP^1)^n$
so it has a $k$-point.
Thus $\eta(Z(k))=\Aff^n(k)$.
\end{proof}

\begin{proof}[Proof of existence in Theorem~\ref{T:cover}]
We use induction on $\dim X$.
We may assume that $X$ is integral
and that $X(k)$ is Zariski dense in $X$;
then $X$ is generically smooth,
and the non-smooth locus $X_\sing$ is of lower dimension.
Let $U_\sing = U \intersect X_\sing$.
The inductive hypothesis gives $\pi_1\colon Y_1 \to X_\sing$
such that $\pi_1(Y_1(k))=U_\sing(k)$.
If we prove the conclusion for $U-U_\sing \subseteq X$,
i.e., if we find $\pi_2\colon Y_2 \to X$ 
such that $\pi_2(Y_2(k))=(U-U_\sing)(k)$,
then the disjoint union $Y_1 \coprod Y_2$
serves as a $Y$ for $U \subseteq X$.
Thus we reduce to the case where $U$ is smooth over $k$.

If $U$ is a finite union of open subvarieties $U_i$,
then it suffices to prove the conclusion for each $U_i \subseteq X$
and take the disjoint union of the resulting $Y$'s.
In particular, we may reduce to the case where $U=X-D$
for some very ample effective divisor $D \subseteq X$.
In other words, we may assume that $X \subseteq \PP^n$
and $U = X \intersect \Aff^n$.

Let $Z \to \PP^n$ be as in Lemma~\ref{L:P^n}.
Define $Y_0$ to make a cartesian diagram
\[
\xymatrix{
Y \ar[rd] \ar@/_/@{..>}[rdd]_{\pi} \\
& Y_0 \ar[r] \ar[d] & Z \ar[d]^\eta \\
& X \ar@{^{(}->}[r] & \PP^n \\
U \ar@{^{(}->}[r] \ar@{^{(}->}[ru] & \Aff^n \ar@{^{(}->}[ru] \\
}
\]
and let $Y \to Y_0$ be a resolution of singularities that is an
isomorphism above the smooth locus of $Y_0$,
so $Y$ is a regular projective variety.
Let $\pi$ be the composition $Y \to Y_0 \to X$.

Suppose that $t \in U(k)$.
Then $Z \to \PP^n$ is smooth above $t$, by choice of $Z$.
So $Y_0 \to X$ is smooth above $t$.
Moreover, $U \to \Spec k$ is smooth,
so $Y_0 \to \Spec k$ is smooth above $t$.
Therefore $Y \to Y_0$ is a local isomorphism above $t$.
Thus $\pi^{-1}(t) \isom \eta^{-1}(t)$,
and the latter has a $k$-point.

On the other hand, if $t \in X(k)-U(k)$,
then $\pi^{-1}(t)$ cannot have a $k$-point,
since such a $k$-point would map to a $k$-point of $Z$ lying
over $t \in \PP^n(k)-\Aff^n(k)$, contradicting the choice of $Z$.

Thus $\pi(Y(k))=U(k)$.
\end{proof}

\begin{remark}
\label{R:curve}
In the special case where $X$ is a regular projective curve
and $U$ is an affine open subvariety of $X$, 
the reductions may be simplified greatly.
Namely, using the Riemann-Roch theorem,
construct a morphism $f \colon X \to \PP^1$
such that $f^{-1}(\infty) = X-U$;
now define $Y_0$ to make a cartesian diagram
\[
\xymatrix{
Y \ar[rd] \ar@/_/@{..>}[rdd]_{\pi} \\
& Y_0 \ar[r] \ar[d] & \calW \ar[d]^\mu \\
& X \ar[r]^f & \PP^1 \\
}
\]
and let $Y$ be a resolution of singularities of $Y_0$.
\end{remark}

\section{Effectivity}

The construction of $Y$ in Theorem~\ref{T:cover} as given
is not effective, because it used Faltings' theorem.
More specifically, in the proof of Lemma~\ref{L:base change}
we know that $C_i'(k)$ is finite but might not know what it is,
so when we reach the last paragraph of the proof, 
we might not know what the finite set $T$ is,
and hence we have no algorithm for computing a good $c$,
where \defi{good} means that 
the images of $c$ and $T - \{0\}$ 
in $k^\times/k^{\times 2}$ do not meet.

\begin{proof}[Existence of an algorithm for Theorem~\ref{T:cover}]
Let $F$ be the (finite) set of $t \in \PP^1(k)$ 
such that $\calV_t$ is not smooth.
Suppose that instead of requiring that $c$ be good,
we require only the effectively checkable condition 
that the images of $c$ and $F$ in $k^\times/k^{\times 2}$ do not meet.
Then the proof of existence in Theorem~\ref{T:cover}
still yields a regular projective variety $Y_c$ 
and a morphism $\pi_c \colon Y_c \to X$,
but it might not have the desired property $\pi_c(Y_c(k))=U(k)$.
Indeed, in the proof of Proposition~\ref{P:Chatelet bundle},
some of the Ch\^atelet surfaces $\calW_t$ other than $\calW_\infty$
may be defined by a reducible degree-$4$ polynomial
and hence may violate the Hasse principle;
thus the conclusion $\mu(\calW(k))=\Aff^1(k)$ 
in Proposition~\ref{P:Chatelet bundle}
must be weakened to $\mu(\calW(k)) \subseteq \Aff^1(k)$,
and this eventually implies $\pi_c(Y_c(k)) \subseteq U(k)$.

On the other hand, an argument of Parshin (see~\cite{Szpiro1985})
shows that Faltings' proof of the Mordell conjecture
can be adapted to give an upper bound on
the {\em size} of each set $C_i'(k)$ in the proof of Lemma~\ref{L:base change}.
Therefore we can compute a bound on $\#T$.
Choose a finite subset $\Gamma \subseteq k^\times$ 
whose image in $k^\times/k^{\times 2}$ is disjoint from the image
of $F$ and has size greater than $\#T$.
Then $\Gamma$ contains at least one good $c$.

Let $Y = \coprod_{c \in \Gamma} Y_c$,
and define $\pi \colon Y \to X$ by $\pi|_{Y_c}=\pi_c$.
Then $\pi(Y(k)) = \Union_{c \in \Gamma} \pi_c(Y_c(k)) = U(k)$
since all terms in the union are subsets of $U(k)$
and some term equals $U(k)$.
\end{proof}

\section{Algorithms for rational points}
\label{S:algorithms}

\begin{proof}[Proof of Theorem~\ref{T:algorithms}(i)]
Suppose we want to know whether the $k$-variety $U$ has a $k$-point.
By passing to a finite open cover, we may assume that $U$ is affine.
Let $X$ be a projective closure of $U$.
Construct $Y \to X$ be as in Theorem~\ref{T:cover}.
Let $Y_1,\ldots,Y_n$ be the connected components of $Y$.
If $Y_i$ is geometrically integral, by assumption we can decide whether 
$Y_i$ has a $k$-point.
If $Y_i$ is not geometrically integral, then since $Y_i$ is regular,
it has no $k$-point.
Thus we can decide whether $Y$ has a $k$-point,
and the answer for $U$ is the same.
\end{proof}

\begin{proof}[Proof of Theorem~\ref{T:algorithms}(ii)]
We want to compute $\#X(k)$.
Apply the algorithm of Theorem~\ref{T:algorithms}(i) to $X$.
If it says that $X$ has no $k$-point, we are done.
Otherwise, search until a $k$-point $P$ on $X$ is found,
and start over with the variety $X - \{P\}$.
If $X(k)$ is finite, this algorithm will eventually terminate.
(This kind of argument was used also in~\cite{Kim2003}.)
\end{proof}

\section{Global function fields}\label{S:function fields}

In this section, we investigate whether the proofs of the previous sections 
carry over to the case where $k$ is a global function field
of characteristic not $2$.

The main issues are 
\begin{enumerate}
\item The two-part paper \cites{CT-Sansuc-SD1987I,CT-Sansuc-SD1987II},
which is key to all our main results, works only over number fields.
But it seems likely that the same proofs work, with at most minor
modifications, over any global field of characteristic not $2$.
\item The proof of Theorem~\ref{T:cover} uses resolution of singularities,
which is not proved in positive characteristic.
Moreover, the proof of Theorem~\ref{T:algorithms} uses Theorem~\ref{T:cover}
so it also is in question.
Without assuming resolution of singularities,
one would obtain the weaker versions of 
Theorem \ref{T:algorithms} and~\ref{T:cover}
in which the word ``regular'' is removed from both.
\end{enumerate}

There are a few other issues, but these can be circumvented,
as we now discuss.

The proof of Proposition~\ref{P:Chatelet}
works for any global function field $k$ of characteristic not $2$:
fix a place $\infty$ of $k$, let $\OO_k$ be the ring of functions
that are regular outside $\infty$, and replace the archimedean and $2$-adic
conditions on $a$ and $b$ by the condition that $a$ and $b$ be squares
in the completion $k_\infty$; then the proof proceeds as before.

The second paragraph of the proof of Lemma~\ref{L:base change}
encounters two problems in positive characteristic: 
first, it needs $\nu_i$ to be separable,
and second, to apply the function field analogue~\cite{Samuel1966} 
of Faltings' theorem it needs $C_i'$ to be non-isotrivial.
As for the first problem, if in Section~\ref{S:Chatelet bundles} 
we choose $\tilde{P}_\infty(w,x)$ to be separable,
then the same will be true of $\tilde{P}_\infty + v^2 \tilde{P}_0$
over $k(v)$,
and the same will be true of the $\nu_i$ in the application of 
Lemma~\ref{L:base change}, since the $\nu_i$ correspond to field extensions
of $k(v)$ contained in the splitting field 
of $\tilde{P}_\infty + v^2 \tilde{P}_0$ over $k(v)$.
As for the second problem, 
the flexibility in the choice of $Q$ in the proof of Lemma~\ref{L:base change}
lets us arrange for $C_i'$ to be non-isotrivial.
Moreover, in this case, one can bound not only the number of $k$-points
on each $C_i'$, but also their height \cite{Szpiro1981}*{\S8, Corollaire~2}.

There is another thing that is better over global function fields $k$
than over number fields.
Namely, by a proved generalization of Hilbert's tenth problem to such $k$
\cites{Pheidas1991,Shlapentokh1992functionfield,Videla1994,Eisentraeger2003},
it is already known that there is no algorithm for deciding 
whether a $k$-variety has a $k$-point.
Therefore, if $k$ is a global function field of characteristic not $2$,
and we assume that \cites{CT-Sansuc-SD1987I,CT-Sansuc-SD1987II} 
works over $k$,
then there is no algorithm for deciding whether a projective
geometrically integral $k$-variety has a $k$-point
(and if we moreover assume resolution of singularities, 
we can add the adjective ``regular'' in this final statement).

\section{Open questions}\label{S:questions}

\begin{enumerate}
\item[(i)]
Can one generalize Remark~\ref{R:many remarks}(f)
to show that to have algorithms in (i) and (ii)
of Theorem~\ref{T:algorithms} for $n$-folds,
it would suffice to be able to decide the existence of rational points
on regular projective geometrically integral {\em $(n+2)$-folds}.
\item[(ii)]
Is there a proof of Proposition~\ref{P:Chatelet}
that does not require such explicit calculations?
\item[(iii)]
Is the problem of deciding whether a smooth projective geometrically
integral {\em hypersurface} over $k$ has a $k$-point
also equivalent to the problem for arbitrary $k$-varieties?
\end{enumerate}

\section*{Acknowledgements} 

I thank Jean-Louis Colliot-Th\'el\`ene for several helpful comments,
Brian Conrad for encouraging me to examine the function field analogue,
Thomas Graber for a remark leading to Proposition~\ref{P:function fields over C},
and J.\ Felipe Voloch and Olivier Wittenberg for suggesting some references.

\end{document}